\documentclass{amsart}

\usepackage{hyperref}
\usepackage{amsthm}
\usepackage{amsmath}
\usepackage{amssymb}
\usepackage{color}
\usepackage{graphicx}
\usepackage{url}

\newcommand{\R}{\mathbb{R}}

\newcommand{\T}{\mathbb{T}}
\newcommand{\TP}{\mathbb{TP}}

\newcommand{\C}{\mathcal{C}}

\newcommand{\B}{\mathcal{B}}

\newcommand{\D}{\mathcal{D}}

\newcommand{\puiseux}{\mathbb{C}\{\!\{t\}\!\}}
\newcommand{\sn}{2^{[n]}}

\DeclareMathOperator{\supp}{supp}
\DeclareMathOperator{\convex}{convex}

\DeclareMathOperator{\tdet}{tdet}
\DeclareMathOperator{\TGr}{TGr}
\DeclareMathOperator{\Dr}{Dr}
\DeclareMathOperator{\Cayley}{Cayley}

\newtheorem{theorem}{Theorem}[section]
\newtheorem{lemma}[theorem]{Lemma}
\newtheorem{proposition}[theorem]{Proposition}
\newtheorem{corollary}[theorem]{Corollary}
\newtheorem{conjecture}[theorem]{Conjecture}
\theoremstyle{definition}
\newtheorem{definition}[theorem]{Definition}
\newtheorem{example*}[theorem]{Example}

\theoremstyle{remark}

\DeclareRobustCommand{\qedify}[1]{%
  \ifmmode \quad\hbox{#1}
  \else
    \leavevmode\unskip\penalty9999 \hbox{}\nobreak\hfill
    \quad\hbox{#1}%
  \fi
}
\newenvironment{example}{\begin{example*}\pushQED{\qedify{$\diamondsuit$}}}{\popQED\end{example*}}

\begin{document}

\begin{abstract}
 In this paper we study general tropical linear spaces locally: For any basis $B$ of the matroid underlying a tropical linear space $L$, we define the local tropical linear space $L_B$ to be the subcomplex of $L$ consisting of all vectors $v$ that make $B$ a basis of maximal $v$-weight. The tropical linear space $L$ can then be expressed as the union of all its local tropical linear spaces, which we prove are homeomorphic to Euclidean space. Local tropical linear spaces have a simple description in terms of polyhedral matroid subdivisions, and we prove that they are dual to mixed subdivisions of Minkowski sums of simplices. Using this duality we produce tight upper bounds for their $f$-vectors. We also study a certain class of tropical linear spaces that we call conical tropical linear spaces, and we give a simple proof that they satisfy Speyer's $f$-vector conjecture.
\end{abstract}

\title{\textsf{Local tropical linear spaces}}
\author{\textsf{Felipe Rinc\'on}}
\address{Mathematics Institute, University of Warwick, Coventry, CV4 7AL, United Kingdom.}
\email{e.f.rincon@warwick.ac.uk}
\thanks{\textsf{University of Warwick, Coventry, UK.}\quad{\tt e.f.rincon@warwick.ac.uk}}
\maketitle

\section{Introduction} 

Tropical linear spaces are one of the most basic objects in tropical geometry. They are polyhedral complexes associated to tropical Pl\"ucker vectors (or valuated matroids), and they are obtained as tropicalizations of classical linear subspaces. Tropical linear spaces play a prominent role in several contexts like the study of tropicalizations of varieties obtained as the image of a linear subspace under a monomial map \cite{tropdisc}, or the study of realizability questions and intersection theory in tropical geometry (see \cite{kristin}, \cite{diagonal}, \cite{realizationspaces}). 

In \cite{speyer}, Speyer studied the combinatorial structure of tropical linear spaces, and in particular, he showed that tropical linear spaces can be described in terms of subdivisions of matroid polytopes (as we describe in Section \ref{seclocal}). He formulated a conjecture on the upper bound for the $f$-vector of a tropical linear space:
\begin{conjecture}[The $f$-vector conjecture \cite{speyer}]\label{fvector}
 If $L$ is an $m$-dimensional tropical linear space in $\R^n$ then $L$ has at most $\binom{n-i-1}{i-1} \binom{n-2i}{m-i}$ faces of dimension $i$ that become bounded after modding out by the lineality space generated by the vector $(1,\dotsc,1) \in \R^n$.
\end{conjecture}
A few special cases of this conjecture were proved in \cite{speyer}, like the case when $i=1$. In \cite{speyerinvariant} Speyer proved it for tropical linear spaces which are realizable over a field of characteristic zero, by studying an interesting matroid invariant arising from the $K$-theory of the Grassmannian. The conjecture is still open in the general case.

Tropical linear spaces have also been studied in connection to Dressians and tropical Grassmannians. In \cite{drawtropicalplanes}, several combinatorial results on tropical planes were developed to study the Dressians $\Dr(3,n)$ and the tropical Grassmannians $\TGr(3,n)$, with an emphasis on the case $n = 7$. In \cite{joswig} some of these results were extended and applied to study the case $n=8$, achieving a combinatorial characterization of all rays in the Dressian $\Dr(3,8)$.

In this paper we study tropical linear spaces locally: For any basis $B$ of the matroid underlying a tropical linear space $L$, we define the local tropical linear space $L_B$ to be the subcomplex of $L$ consisting of all vectors $v \in L$ that make $B$ a basis of maximal $v$-weight. A more precise definition is given in Section \ref{seclocal}. The space $L_B$ consists then of all cells of $L$ that are ``near'' the vertex $e_B$ of the underlying matroid polytope $\Gamma(M)$. Our definition generalizes the notion of local Bergman fan given by Feichtner and Sturmfels in \cite{feichtnersturmfels} to the non-constant coefficient case.

The tropical linear space $L$ is the union of all its local tropical linear spaces. In Section \ref{seclocal} we prove that they are all homeomorphic to Euclidean space by giving an explicit piecewise linear homeomorphism. This leads to an effective algorithm for projecting points onto general tropical linear spaces, also described by Corel in \cite{corel}. Our explicit description could also be used for efficiently computing local tropical linear spaces and thus general tropical linear spaces, by generalizing the approach given in \cite{computing} to the non-constant coefficient case.

In Section \ref{secmixed} we study the combinatorics of local tropical linear spaces, and we use the Cayley trick to prove that they are combinatorially dual to mixed subdivisions of Minkowski sums of simplices. This allows us to produce tight upper bounds on their $f$-vectors. 

In Section \ref{secconical} we study a certain class of tropical linear spaces that we call conical tropical linear spaces, which correspond to polyhedral matroid subdivisions in which all maximal faces share a common vertex. In the one dimensional case, these tropical linear spaces correspond exactly to caterpillar trees. Conical tropical linear spaces turn out to be realizable and thus they satisfy Speyer's $f$-vector conjecture, but we give a purely combinatorial proof of that fact. As an application, we give an independent proof of a result of Herrmann, Joswig, and Speyer, which appeared as a conjecture in a first version of \cite{joswig}.

\section{Local tropical linear spaces}\label{seclocal}

In this section we define local tropical linear spaces and prove that they are homeomorphic to Euclidean space. We start with some general background on tropical Pl\"ucker vectors and tropical linear spaces.

Let $m \leq n$ be nonnegative integers. We will use the notation $[n]:=\{1,\dotsc,n\}$, $2^{[n]}:=$ power set of $[n]$, and $\binom{[n]}{m}:=\{S \in 2^{[n]} \mid \left|S\right| = m\}$. Let $\T := (\R \cup \{\infty\}, \min, +)$ be the \textbf{tropical semiring} of real numbers \emph{including infinity}. A vector $p \in \T^{\binom{[n]}{m}}$ is called a \textbf{tropical Pl\"ucker vector} of rank $m$ if it satisfies the tropical Pl\"ucker relations, that is, for any $S, T \in \sn$ satisfying $|S| = m-1$ and $|T|= m+1$, the minimum
\begin{equation}\label{eqplucker}
 \min_{i \in T \setminus S} (p_{S \cup i} + p_{T-i})
\end{equation}
is achieved at least twice (i.e., for at least two different values of $i$) or it is equal to $\infty$. It follows that the \textbf{support} $\supp(p) := \{ B \in \binom{[n]}{m} \mid p_B \neq \infty \}$ of $p$ is the collection of bases of matroid over $[n]$, called the \textbf{underlying matroid} of $p$. In the literature, tropical Pl\"ucker vectors have also been studied under the name of \textbf{valuated matroids}, but using the opposite sign convention to ours \cite{valuated}. The space of all tropical Pl\"ucker vectors is called the \textbf{Dressian}, and it is denoted as $\Dr_{m,n} \subseteq \T^{\binom{[n]}{m}}$.

Tropical Pl\"ucker vectors can be described in term of polyhedral subdivisions in the following way. To any collection $\mathcal{S}$ of subsets of $[n]$ we can associate a $0/1$ polytope $\Gamma(\mathcal{S}):= \convex \{e_S \mid S \in \mathcal{S} \} \subseteq \R^n$, where $e_S:= \sum_{i \in S} e_i$. Matroids can be easily characterized from this point of view (see \cite{ggms}): A collection $\mathcal{S} \subseteq 2^{[n]}$ is the collection of bases $\B(M)$ of a matroid $M$ over the ground set $[n]$ if and only if its associated polytope $\Gamma(\mathcal{S})$ has only edges of the form $e_i - e_j$ for $i,j \in [n]$ distinct. In this case, the polytope $\Gamma(M) := \Gamma(S)$ is called a \textbf{matroid polytope}. Its edges correspond to pairs of bases $A, B \in \B(M)$ such that $A = B -i \cup j$ for $i \neq j$, called \textbf{adjacent bases}.

A \textbf{subdivision} of a polytope $P$ is a set of polytopes $S=\{P_1, \ldots, P_m\}$, whose vertices are vertices of $P$, such that $P_1 \cup \cdots \cup P_m = P$, and if an intersection $P_i \cap P_j$ is nonempty then it is a proper face of both $P_i$ and $P_j$. Any vector $p \in \T^{\binom{[n]}{m}}$ induces a polytopal subdivision of $\Gamma := \Gamma(\supp(p))$ as follows. The vector $p \in \T^{\binom{[n]}{m}}$ can be thought of as a height function on the vertices of $\Gamma$, giving rise to the ``lifted polytope'' $\Gamma(p):= \convex \{ (e_S, p_S) \in \R^{n+1} \mid S \in \supp(p)\}$. Projecting the lower facets of $\Gamma(p)$ (i.e., its facets whose outward normal vector has a negative $(n+1)$st coordinate) back to $\R^n$, we get a polytopal subdivision $\mathcal{D}_p$ of $\Gamma$, called the \textbf{regular subdivision} induced by $p$. 

Tropical Pl\"ucker vectors admit a beautiful characterization in this language (see \cite{speyer, isotropical}): A vector $p \in \T^{\binom{[n]}{m}}$ is a tropical Pl\"ucker vector if and only if the regular subdivision $\mathcal{D}_p$ is a \textbf{matroid polytope subdivision}, i.e., it is a subdivision of a matroid polytope into matroid polytopes.

Let $p \in \T^{\binom{[n]}{m}}$ be a tropical Pl\"ucker vector. For $S \in \binom{[n]}{m+1}$, let $c_S \in \T^n$ be defined by
 \begin{equation}\label{defcircuit}
  (c_S)_i := 
  \begin{cases}
   p_{S-i} & \text{ if $i \in S$,} \\
   \infty & \text{ otherwise.}
  \end{cases}
 \end{equation}
If $c_S$ is not equal to $\vec{\infty} := (\infty, \dotsc, \infty) \in \T^n$, any vector of the form $c_S + \lambda \cdot \textbf{1}$ with $\lambda \in \R$ is called a \textbf{(valuated) circuit} of $p$, where $\textbf{1} := (1,\dotsc,1) \in \R^n$. Its support $\supp(c_S + \lambda \cdot \textbf{1}) = \supp(c_S) := \{ i \in [n] \mid (c_S)_i \neq \infty \}$ is a circuit of the underlying matroid $M$ of $p$. It can be proved (see \cite{murotatamura, isotropical}) that if two circuits of $p$ have the same support then they differ by a scalar multiple of the vector $\textbf{1}$, that is, the two circuits are the same in tropical projective space $\TP^{n-1} := (\T^n - \vec{\infty}) / \R \cdot \textbf{1}$.

For any basis $B \subseteq [n]$ of $M$ and any $e \in [n]$ there is a unique circuit of $M$ contained in $B \cup e$ (containing the element $e$), which is called the \textbf{fundamental circuit} $C(e,B)$ of $e$ over $B$. If the support of $c_S$ is a fundamental circuit of $M$ over some basis $B$ then we say that $c_S$ is a \textbf{(valuated) fundamental circuit} of $p$ over the basis $B$.  

Valuated circuits satisfy the following valuated elimination property, which generalizes the classical elimination axiom for circuits of a matroid. For any two vectors $d, e \in \T^n$, we denote by $\min(d,e) \in \T^n$ the corresponding vector of coordinate-wise minima. 
\begin{proposition}[\cite{murotatamura}]\label{valuatedelimination}
 If $d, e \in \T^n$ are two circuits of $p$ and $a,b \in [n]$ are such that $d_a < e_a$ and $d_b = e_b \neq \infty$, then there exists a circuit $f \in \T^n$ of $p$ satisfying $f_b = \infty$, $f_a = d_a$, and $f \geq \min(d,e)$.
\end{proposition}

Two vectors $x, y \in \T^n$ are said to be \textbf{tropically orthogonal}, denoted by $x \top y$, if the minimum $\min (x_1 + y_1, \dotsc , x_n + y_n)$ is achieved at least twice (or it is equal to $\infty$). If $X \subseteq \T^n$ then its \textbf{tropically orthogonal set} is $X^\top := \{ y \in \T^n \mid y \top x \text{ for all } x \in X \}$.

If $p \in \T^{\binom{[n]}{m}}$ is a tropical Pl\"ucker vector, denote by $\C(p) \subseteq \T^n$ the set of all circuits of $p$. The space $L(p) := \C(p)^\top \subseteq \T^n$ is called the \textbf{tropical linear space} associated to $p$. Tropical linear spaces were introduced and studied by Speyer in \cite{speyer}. 

As we briefly discussed in the introduction, tropical linear spaces play a very important role in tropical geometry. Let $K := \puiseux$ be the field of Puiseux series, and consider the $n$-dimensional vector space $V:=K^n$. Suppose $W$ is an $m$-dimensional linear subspace of $V$ with Pl\"ucker coordinates $P \in K^{\binom{[n]}{m}}$. Let $p \in \T^{\binom{[n]}{m}}$ be the valuation of the vector $P$. Since $P$ satisfies the Pl\"ucker relations, the vector $p$ is a tropical Pl\"ucker vector. In this case, the tropicalization of the linear space $W$ is precisely the tropical linear space $L(p)$ \cite{speyer}.

We now define the main object of study in our paper.
\begin{definition}
 Let $p \in \T^{\binom{[n]}{m}}$ be a tropical Pl\"ucker vector with underlying matroid $M$. If $B$ is a basis of $M$ and $v \in \R^n$, we define the \textbf{$v$-weight} of $B$ (with handicap $p$) to be $w_p(v,B) := -p_B + \sum_{i \in B} v_i$. For any $v \in \R^n$, the collection of bases of $M$ with maximal $v$-weight is the collection of bases of a matroid $M_v$ on the ground set $[n]$. In the notation we introduced above for regular matroid subdivisions, $M_v$ is the matroid corresponding to the face of $\mathcal{D}_p$ obtained as the projection of the face of $\Gamma(p)$ that maximizes the dot product with the vector $(v, -1) \in \R^{n+1}$. 
 
 Now, for any basis $B$ of $M$, denote by $\Sigma_B$ the set of vectors $v \in \R^n$ such that $M_v$ contains the basis $B$. The \textbf{local tropical linear space} $L(p)_B$ is defined as $L(p)_B := L(p) \cap \Sigma_B$.
\end{definition}

Suppose $p \in \R^{\binom{[n]}{m}}$ is a tropical Pl\"ucker vector (with no coordinates equal to $\infty$), so its underlying matroid is the uniform matroid $U_{m,n}$. The vector $p$ induces then a regular matroid subdivision $\mathcal{D}_p$ of the hypersimplex $\Delta_{m,n}$. It was shown in \cite{speyer} that the tropical linear space $L(p) \cap \R^n$ consists of all vectors $v \in \R^n$ such that $M_v$ is a loopless matroid. In particular, $L(p) \cap \R^n$ is a polyhedral complex dual to the faces of $\mathcal{D}_p$ that correspond to loopless matroids. If $B$ is a basis of the uniform matroid $U_{m,n}$ then the local tropical linear space $L(p)_B$ consists of the cells of $L(p) \cap \R^{n}$ which are dual to faces of $\mathcal{D}_p$ that correspond to loopless matroids and \emph{contain the vertex $e_B$}. All these results hold more generally for any tropical Pl\"ucker vector in $\T^{\binom{[n]}{m}}$, as we will show in Corollary \ref{dualloopless}.

\begin{example}\label{exlocal}
Let $n=4$, $m=2$, and consider the vector $p \in \R^{\binom{[n]}{m}}$ defined as
\[
 p_S := 
 \begin{cases}
  1 & \text{ if $S=12$ or $S=34$,} \\
  0 & \text{ if $S=13$ or $S=14$ or $S=23$ or $S=24$.}
 \end{cases}
\]
The hypersimplex $\Delta_{2,4}$ is the convex hull of all $0/1$ vectors in $\R^4$ having exactly two coordinates equal to $1$. This polytope lives in the $3$-dimensional hyperplane defined by $x_1 + x_2 + x_3 + x_4 = 2$, and is in fact a regular octahedron. The regular subdivision $\mathcal{D}_p$ induced by $p$ is consists of two square pyramids meeting at their base, as depicted in the left of Figure \ref{figlinearspace}. Since all faces of this subdivision are matroid polytopes, the vector $p$ is a tropical Pl\"ucker vector. The tropical linear space $L(p) \cap \R^n$ is then dual to all the faces of this subdivision which correspond to loopless matroids, as drawn in green and red on the right side of Figure \ref{figlinearspace}. The local tropical linear space around the basis $B = 14$ consists of the cells of $L(p) \cap \R^n$ that are dual to faces of $\mathcal{D}_p$ containing the vertex $e_{14}$, which are precisely the green cells in the picture. Note that this local tropical linear space is 
homeomorphic to $\R^2$ (before modding out by the lineality space generated by the vector $\textbf{1} \in \R^4$), as we will later generalize in Theorem \ref{localtropical}.
\begin{figure}[ht]
\begin{center}
\includegraphics[scale=0.6]{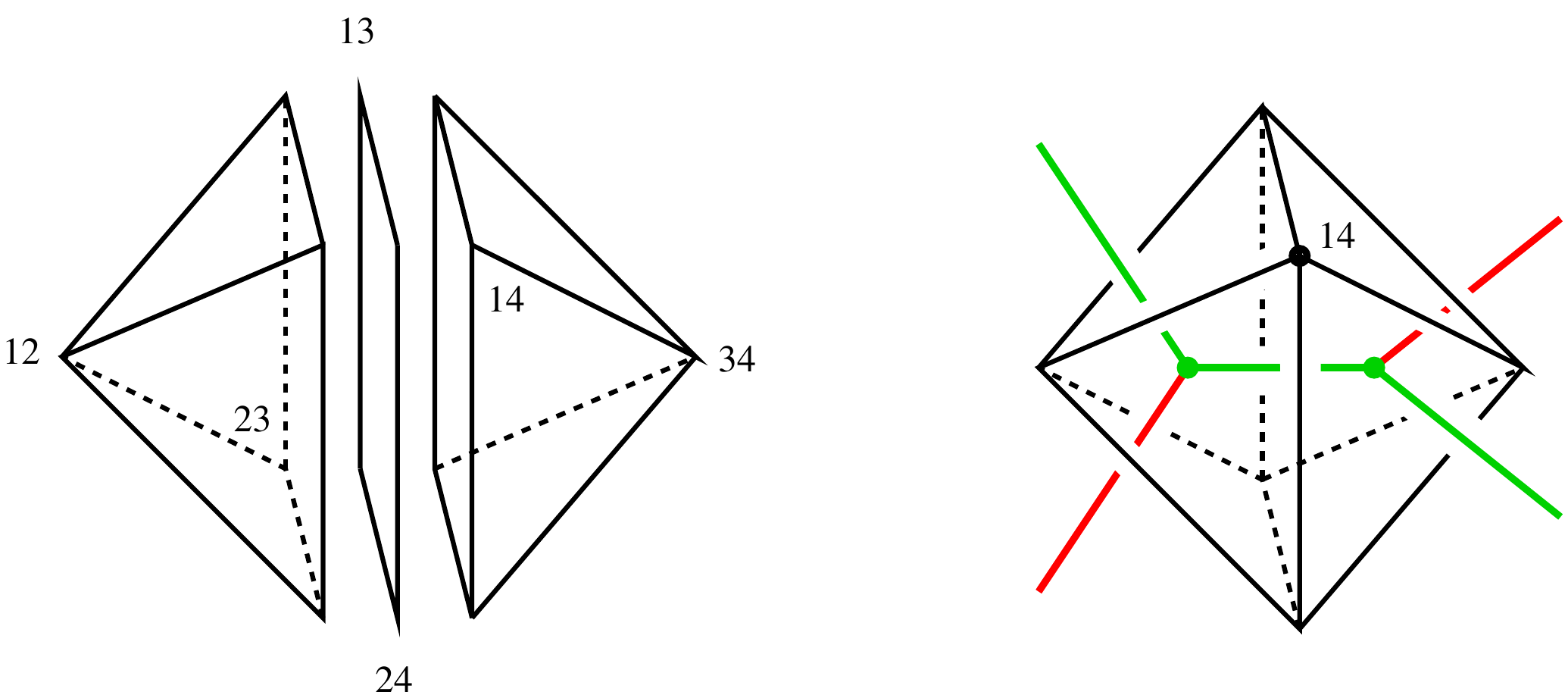}
\caption{At the left, the regular matroid subdivision induced by the tropical Pl\"ucker vector of Example \ref{exlocal}. At the right, its associated tropical linear space inside the hypersimplex, with the green subcomplex corresponding to the local tropical linear space around the basis $\{1,4\}$.}
\label{figlinearspace}
\end{center}
\end{figure}
\end{example}

Any tropical linear space $L(p) \cap \R$ is the union of all its tropical linear spaces $L(p)_B$, so we can attempt to understand tropical linear spaces by studying them locally.

\begin{lemma}\label{onlyfundamental}
 Let $p \in \T^{\binom{[n]}{m}}$ be a tropical Pl\"ucker vector, and let $B$ be a basis of its underlying matroid $M$. For any $v \in \Sigma_B$, $v$ is in the local tropical linear space $L(p)_B$ if and only if $v$ is tropically orthogonal to all \emph{fundamental circuits} of $p$ over the basis $B$.
\end{lemma}
\begin{proof}
 By definition, if $v \in L(p)$ then $v$ is tropically orthogonal to all circuits of $p$, in particular to all fundamental circuits over the basis $B$. To prove the converse, assume by contradiction that $v \in \Sigma_B$ is tropically orthogonal to all fundamental circuits of $p$ over the basis $B$ but $v$ is not in $L(p)_B$, thus $v$ is not in $L$. Let $d$ be a circuit of $p$ which is not tropically orthogonal to $v$, and take $d$ such that the circuit $D := \supp (d)$ of the underlying matroid $M$ contains as few elements outside of $B$ as possible. Since $D$ is not a fundamental circuit over $B$, it contains at least two elements not in $B$. Let $a \in D$ be the unique element such that $d_a + v_a = \min \{d_i + v_i \mid i \in [n]\}$, and let $b \in D - B$ be different from $a$. After adding a suitable scalar multiple of the vector $\textbf{1}$, we can assume that $d_b = p_B$. Let $S := B \cup b$, and consider the circuit $e := c_S$ of $p$, as defined in Equation \eqref{defcircuit}. Its support $\supp(e)$ 
is the fundamental circuit $E := C(b, B)$ of $b$ over $B$ in the matroid $M$. Since $v \in \Sigma_B$, it follows that 
 \begin{equation}\label{eqineq}
  e_b + v_b = p_B + v_b \leq p_{S - i} + v_i = e_i + v_i
 \end{equation}
 for any $i \in E$. Note that in fact this inequality holds for any $i \in [n]$. We thus have $d_a + v_a < d_b + v_b = p_B + v_b \leq e_a + v_a$, so $d_a < e_a$. Applying Proposition \ref{valuatedelimination}, we get that there is a circuit $f$ of $p$ such that $f_b = \infty$, $f_a = d_a$, and $f \geq \min(d,e)$, where $\min(d,e) \in \T^n$ denotes the corresponding vector of coordinate-wise minima. We have
 \begin{equation*}
  f_a + v_a = d_a + v_a < d_i + v_i
 \end{equation*}
 for any $i \in [n]$ different from $a$, and therefore
 \begin{equation*}
  f_a + v_a < d_b + v_b = e_b + v_b \leq e_i + v_i
 \end{equation*}
 for any $i \in [n]$ (the last inequality in the previous line comes from \eqref{eqineq}). Since $f \geq \min(d,e)$, these last two inequalities imply that $f_a + v_a < f_i + v_i$ for any $i \in [n]$ different from $a$, so $\min_{i \in [n]}(f_i + v_i)$ is achieved only once at $i = a$. But this means that $f$ is a circuit of $p$ which is not tropically orthogonal to $v$, and whose support $F:= \supp(f)$ has fewer elements outside of $B$ than the circuit $D$ (since $F \subseteq D \cup E - b \subseteq B \cup D - b$), which contradicts our choice of $d$.
\end{proof}

Lemma \ref{onlyfundamental} can be stated in polyhedral terms in the following way. It generalizes to arbitrary tropical Pl\"ucker vectors in $\T^{\binom{[n]}{m}}$ the combinatorial description of tropical linear spaces given by Speyer in \cite{speyer}.

\begin{corollary}\label{dualloopless}
 Let $p \in \T^{\binom{[n]}{m}}$ be a tropical Pl\"ucker vector with underlying matroid $M$. A vector $v \in \R^n$ is in the tropical linear space $L(p)$ if and only if $M_v$ is a loopless matroid. In particular, $L(p) \cap \R^n$ is a polyhedral complex dual to the faces of the regular subdivision $\mathcal{D}_p$ that correspond to loopless matroids. 
\end{corollary}
\begin{proof}
 Suppose $B$ is a basis of $M$ and $a \in [n] \setminus B$. The support of the (valuated) circuit $c := c_{B \cup a} \in \T^n$ of $p$ is the fundamental circuit of $a$ over the basis $B$. A vector $v \in \R^n$ is tropically orthogonal to $c$ if and only if $\min_{b \in B \cup a} \, p_{B \cup a - b} + v_b$ is achieved at least twice. Equivalently, $v$ is tropically orthogonal to $c$ if and only if $\max_{b \in C(a,B)} \, w_p(v,B \cup a - b)$ is achieved at least twice, where $C(a,B)$ denotes the fundamental circuit of $a$ over the basis $B$.
 
 Now, let $v \in \R^n$ and take $B$ a basis of $M$ of maximal $v$-weight. According to Lemma \ref{onlyfundamental}, $v \in L(p)$ if and only if $v$ is tropically orthogonal to all fundamental circuits of $p$ over the basis $B$. Our discussion above implies that this is the case if and only if for any $a \in [n] \setminus B$ there exists $b \in B$ such that $B \cup a - b$ is also a basis of maximal $v$-weight. It follows that $v \in L(p)$ if and only if the matroid $M_v$ has no loops, as desired.  
\end{proof}

We use Lemma \ref{onlyfundamental} to construct a piecewise linear homeomorphism between $m$-dimensional Euclidean space and the local tropical linear space $L(p)_B$, generalizing Theorem 4.2 in \cite{feichtnersturmfels}.
\begin{theorem}\label{localtropical}
 Any $m$-dimensional local tropical linear space $L(p)_B$ is homeomorphic to $\R^m$. More specifically, if $B = \{ b_1, \dotsc , b_m\} \subseteq [n]$ then the function $f_B:\R^m \to \R^n$ sending a vector $x = (x_1,\dotsc, x_m) \in \R^m$ to the vector $f_B(x) \in \R^n$ defined by
 \[
  (f_B(x))_i := 
  \begin{cases}
   x_j & \text{if $i = b_j$ for some $j$,} \\
   \displaystyle{\min_{b_j \in C(i,B) - i}} x_j + p_{B \cup i - b_j} - p_B & \text{if $i \in [n] \setminus B$;}
  \end{cases}
 \]
 is a piecewise linear homeomorphism between $\R^m$ and $L(p)_B$.
\end{theorem}
\begin{proof}
 We first prove that the image of $f_B$ lies in $\Sigma_B$. Suppose $x \in \R^m$, and let $v := f_B(x) \in \R^n$. Assume by contradiction that there is a basis $A$ of $M$ such that $w_p(v,A) > w_p(v,B)$, and take $A$ such that $|A \setminus B|$ is minimal. Let $a \in A \setminus B$, and define $S := A -a$, $T := B \cup a$. Since $p$ is a tropical Pl\"ucker vector, the minimum in \eqref{eqplucker} is attained at least twice, so there exists a $b \in B$ such that $p_A + p_B \geq p_{A-a\cup b} + p_{B-b \cup a}$. Subtracting $\sum_{i \in A} v_i + \sum_{j \in B} v_j$ on both sides we get $w_p(v,A) + w_p(v,B) \leq w_p(v,A-a\cup b) + w_p(v, B-b\cup a)$. But the definition of $f_B$ implies that $w_p(v, B-b\cup a) \leq w_p(v, B)$, so it follows that $w_p(v,A) \leq w_p(v,A-a\cup b)$, contradicting our choice of $A$.
 
 Now, it follows directly from the definition that any vector in the image of $f_B$ is tropically orthogonal to all fundamental circuits of $p$ over the basis $B$, so Lemma \ref{onlyfundamental} ensures that the image of $f_B$ lies in $L(p)_B$. Also, $f_B$ is clearly an injective function. Moreover, if $v$ is any vector in $L(p)_B$ then for any $i \in [n] \setminus B$ we have
 \[
  \min_{b \in C(i,B)} p_{B \cup i - b} + v_b = p_B + v_i, 
 \]
 so it follows that $f_B$ is surjective onto $L(p)_B$.
\end{proof}

Together with the description of tropical linear spaces as tropical polytopes given in \cite{murotatamura, isotropical}, the homeomorphism described in Theorem \ref{localtropical} can be used to show that the tropical projection of a point $y \in \Sigma_B$ onto the tropical linear space $L$ is given by the map
\[
(\pi(y))_i = 
\begin{cases}
y_i & \text{if } i \in B,\\
\displaystyle{\min_{j \in C(i,B) - i}} y_j + p_{B \cup i - j} - p_B & \text{if $i \in [n] \setminus B$.}
\end{cases}
\]
This result is also described by Corel in \cite[Theorem 6]{corel}. Theorem \ref{localtropical} could also be used for efficiently computing local tropical linear spaces, and thus general tropical linear spaces, by generalizing the approach given in \cite{computing} to the non-constant coefficient case.

\section{Face posets and mixed subdivisions}\label{secmixed}

In this section we study the combinatorial structure of local tropical linear spaces, by relating them to mixed subdivisions of Minkowski sums of simplices. 

Let $p \in \T^{\binom{[n]}{m}}$ be a tropical Pl\"ucker vector with underlying matroid $M$, and let $\D := \mathcal{D}_p$ be the regular matroid subdivision induced by $p$ on the matroid polytope $\Gamma := \Gamma(M)$. Let $B$ be any basis of $M$. As discussed above, the local tropical linear space $L(p)_B$ is a polyhedral complex dual to the faces of $\D$ that contain the vertex $e_B$ and correspond to loopless matroids, so in order to study this local tropical linear space it is enough to study how the subdivision $\D$ looks ``around'' the vertex $e_B$. 

Denote by $\Gamma_B$ the subpolytope of $\Gamma$ obtained as the convex hull of all vertices adjacent to $e_B$, i.e., $\Gamma_B := \convex \{ e_A \mid A \in \B(M) \text{ and } |A \setminus B| = 1 \}$. Note that $\Gamma_B$ is the intersection of $\Gamma$ with the affine hyperplane $h_B := \{ \textbf{x} \in \R^n \mid \sum_{i \in B} x_i = m-1 \}$. The polytope $\Gamma_B$ is called the \textbf{vertex figure} of $\Gamma$ around the vertex $e_B$ (or around the basis $B$). The subdivision of $\Gamma_B$ obtained by intersecting the subdivision $\D$ with the polytope $\Gamma_B$ is called the \textbf{local subdivision} induced by $p$ (or induced by $\D$) around the basis $B$, and it is denoted by $\D_B$. 

For any $S \subseteq [n]$, consider the simplex $\Delta_S := \convex \{e_i \mid i \in S \} \subseteq \R^n$. If $j \in [n] \setminus B$, denote $I_j := \left\{ i \in B \mid B - i \cup j \in \B(M) \right\}$. We then have
\begin{align*}
 \Gamma_B &= \convex \{ e_B - e_i + e_j \mid j \in [n] \setminus B \text{ and } i \in I_j \} \\
 &\cong \convex\{ e_j + \Delta_{I_j} \mid j \in [n] \setminus B \}.
\end{align*}
This last polytope is called the \textbf{Cayley embedding} of the simplices $\{ \Delta_{I_j} \mid j \in [n] \setminus B \}$, and we will denote it by $\Lambda = \Cayley\{\Delta_{I_j} \mid j \in [n] \setminus B\} \subseteq \R^n$.

The \textbf{Cayley trick} allows us to encode any subdivision of $\Lambda$ using a mixed subdivision of the Minkowski sum of the simplices $\{ \Delta_{I_j} \mid j \in [n] \setminus B \}$. We now describe this procedure in our particular case, for a more general theory the reader is invited to see \cite{cayley} or \cite[Section 14]{permutohedra}. In order to simplify our notation, we assume without loss of generality that $B=\{1,\dotsc,m\}$. Suppose $\mathcal{S}$ is a subdivision of $\Lambda \subseteq \R^n$. Let $o := \frac{1}{n-m} (e_1 + \dotsc + e_{n-m}) \in \R^{n-m}$, and denote by $\R^m \times \{o\}$ the affine subspace of $\R^n$ consisting of all vectors whose last $n-m$ coordinates are exactly $o$. Intersecting the subdivision $\mathcal{S}$ with the subspace $\R^m \times \{o\} \cong \R^m$ we get a mixed subdivision of $\frac{1}{n-m} \cdot \sum_{j \in [n] \setminus B}\Delta_{I_j}$, which after scaling can be thought of as a mixed subdivision $\mathcal{M}$ of the Minkowski sum $\sum_{j \in [n] \setminus B}
\Delta_{I_j}$. Note that this procedure defines a bijection between the faces of $\mathcal{M}$ of dimension $d$ and the faces of $\mathcal{S}$ of dimension $d + (n-m-1)$ that are not contained in any of the subspaces $\R^m \times \{x_j = 0\}$, for $j \in [n] \setminus B$. Moreover, this bijection preserves inclusion, so it is an isomorphism between the face poset of the mixed subdivision $\mathcal{M}$ and the subposet of the face poset of the subdivision $\mathcal{S}$ consisting of all faces which are not contained in any $\R^m \times \{x_j = 0\}$.

If we apply the Cayley trick to the subdivision $\D_B$ of the polytope $\Gamma_B \cong \Lambda$, we obtain a mixed subdivision $\mathcal{M}_B$ of $\sum_{j \in [n] \setminus B}\Delta_{I_j}$ whose face poset is isomorphic to the subposet of the face poset of $\D_B$ consisting of all faces which are not contained in any $\R^m \times \{x_j = 0\}$, i.e., all faces of $\D_B$ that correspond to loopless matroids in the subdivision $\D$. We have thus proved the following result.

\begin{proposition}\label{propdual}
 The local tropical linear space $L(p)_B$ is combinatorially dual to the mixed subdivision $\mathcal{M}_B$ of the Minkowski sum $\sum_{j \in [n] \setminus B}\Delta_{I_j}$.
\end{proposition}

In the case when the underlying matroid of the tropical Pl\"ucker vector $p$ is the uniform matroid $U_{m,n}$, it was proved in \cite{kapranov} and \cite{joswig} that any regular subdivision of the vertex figure $\Gamma_B$ is induced by a regular matroid subdivision of the hypersimplex $\Delta_{m,n}$. Their results extend easily to our more general setup, where we impose no restrictions on the underlying matroid of $p$.

\begin{proposition}[{\cite[Corollary 1.4.14]{kapranov}, \cite[Proposition 4]{joswig}}]\label{existence}
 For any regular subdivision $\mathcal{S}$ of the vertex figure $\Gamma_B$, there is a regular matroid subdivision $\D$ of some matroid polytope $\Gamma'$ with the same vertex figure as $\Gamma$ (i.e. $\Gamma'_B = \Gamma_B$) that restricts to the subdivision $\mathcal{S}$ when intersected with $\Gamma_B$.
\end{proposition}

We will give some details about the proof of Proposition \ref{existence} and the description of the matroid polytope $\Gamma'$ later in Section \ref{secconical}.

It follows from the two previous propositions that studying the combinatorics of local tropical linear spaces is equivalent to studying the combinatorics of mixed subdivisions of Minkowski sums of simplices, as stated by the following theorem.

\begin{theorem}\label{thmdual}
 A poset $P$ is isomorphic to the face poset of an $m$-dimensional local tropical linear space in $\R^n$ if and only if $P$ is isomorphic to the dual face poset of a coherent mixed subdivision of a Minkowski sum of simplices $\sum_{j = 1}^{n-m}\Delta_{I_j} \subseteq \R^m$, where $I_j \subseteq [m]$ for all $j$.
\end{theorem}

This duality between local tropical linear spaces and mixed subdivisions of Minkowski sums of simplices can be used to get a bound on the $f$-vector of local tropical linear spaces.

\begin{proposition}\label{localfvector}
 The number of $i$-dimensional faces of an $m$-dimensional local tropical linear space in $\R^n$ which become bounded after modding out by the lineality space generated by the vector $(1, \dotsc, 1)$ is at most
\begin{equation*}
 \binom{n-i-1}{ n-m-i, \, i-1,\, m-i} = \binom{n-i-1}{i-1} \binom{n-2i}{m-i},
\end{equation*}
 and the number of $i$-dimensional faces without any boundedness constraint is at most
\begin{equation*}
 \frac{n-m}{n-i} \cdot \binom{n-1}{n-m, \, i-1, \, m-i } = \binom{n-i-1}{m-i} \binom{n-1}{i-1}.
\end{equation*}
Furthermore, for any $m$ and $n$ there is a local tropical linear space that achieves all these bounds.
\end{proposition}
\begin{proof}
 By Proposition \ref{propdual}, the $i$-dimensional (bounded) faces of a local tropical linear space $L(p)_B$ are in correspondence with the (interior) faces of codimension $i-1$ in the associated mixed subdivision $\mathcal{M}_B$ of $\sum_{j \in [n] \setminus B}\Delta_{I_j}$. The maximum number of faces is attained when the mixed subdivision $\mathcal{M}_B$ is a fine mixed subdivision of $(n-m) \cdot \Delta_{m-1}$, so the result follows by substituting $s=n-m$, $r=m$, and $k=m-i$ in the following lemma. The existence of a tropical linear space satisfying these bounds follows from Proposition \ref{existence}.
\end{proof}

\begin{lemma}\label{fvectormixed}
 The number of $k$-dimensional interior faces in any fine mixed subdivision of $s \cdot \Delta_{r-1}$ is equal to 
 \begin{equation}\label{interiorfaces}
  \binom{s-1+k}{s-r+k, \, r-1-k, \, k},
 \end{equation}
 and the total number of $k$-dimensional faces is equal to 
 \[
  \frac{s}{s+k} \cdot \binom{r+s-1}{s, \, r-1-k, \, k}.
 \]
\end{lemma}
\begin{proof}
 Interior faces of dimension $k$ in a fine mixed subdivision of $s\cdot \Delta_{r-1}$ are in correspondence with interior faces of dimension $k+s-1$ in an associated triangulation of the product $\Delta_{r-1}\times \Delta_{s-1}$. Products of simplices are equidecomposable polytopes, that is, all its triangulations have the same f-vector. Moreover, since faces of a product of simplices are also product of simplices, it follows that all triangulations of a product of simplices have the same number of interior faces in each dimension. The number of interior faces in a triangulation of a product of simplices was studied in \cite[Corollary 25]{develin} in connection to tropical polytopes, from which \eqref{interiorfaces} follows.
 
 The total number of faces can be computed by adding interior faces over all faces of $s\cdot \Delta_{r-1}$. For any $1 \leq l \leq r$, there are $\binom{r}{l}$ faces of $s\cdot \Delta_{r-1}$ isomorphic to $s\cdot \Delta_{l-1}$, so the total number of $k$-dimensional faces in a fine mixed subdivision of $s\cdot \Delta_{r-1}$ is equal to
 \begin{align*}
  \sum_{l=1}^r \binom{r}{l} \binom{s-1+k}{s-l+k, \, l-1-k, \, k} &= \sum_{l=1}^r \binom{r}{l} \binom{s-1+k}{k} \binom{s-1}{s-l+k}\\
  &= \binom{s-1+k}{k} \cdot \sum_{l=1}^r \binom{r}{l} \binom{s-1}{s+k-l}\\
  &= \binom{s-1+k}{k} \cdot \binom{r+s-1}{s+k}\\
  &= \frac{s}{s+k} \cdot \binom{r+s-1}{s, \, r-1-k, \, k},
 \end{align*}
 as desired.
\end{proof}

Any regular matroid subdivision (and thus any tropical linear space) can be encoded by the sequence of mixed subdivisions arising from the different local subdivisions around all the vertices of its underlying matroid polytope. It would be very interesting to clarify exactly how these different mixed subdivisions are related to each other.

\section{Conical tropical linear spaces}\label{secconical}

In this section we study a certain class of tropical linear spaces that we call conical tropical linear spaces, and we give a simple proof that they satisfy the $f$-vector conjecture. We start by describing more in depth the relation stated in Proposition \ref{existence} between regular subdivisions of the vertex figure $\Gamma_B$ and regular matroid subdivisions of the matroid polytope $\Gamma'$, following the ideas in \cite{joswig}.

Suppose $\Gamma = \Gamma(M)$ is a matroid polytope, and let $B$ be a basis of the matroid $M$. The vertex figure $\Gamma_B \cong \convex\{ e_j + \Delta_{I_j} \mid j \in [n] \setminus B \}$ is a subpolytope of the product of simplices $\Delta_{[n]\setminus B} \times \Delta_B$. By definition, a regular subdivision $\mathcal{S}$ of $\Gamma_B$ is obtained by lifting its vertices to some heights and then projecting back the lower faces of the resulting polytope. This set of heights on the vertices of $\Gamma_B$ can thus be encoded as a matrix $V \in \T^{B \times ([n]\setminus B)}$ (taking the entry $V_{i,j} = \infty$ if the corresponding vertex $e_j+e_i$ is not in $\Gamma_B$). The \textbf{augmented matrix} of $V$ is the $m \times n$ matrix $\bar{V}$ whose submatrix consisting of the columns indexed by $B$ is equal to the tropical identity matrix of size $m$ (i.e., the $m \times m$ matrix with zeroes on the diagonal and $\infty$ on the rest of its entries), and whose submatrix consisting of the columns indexed by 
$[n] \setminus B$ is equal to $V$. Define the vector $\tau_V \in \T ^{\binom{[n]}{m}}$ as $(\tau_V)_A := \tdet (\bar{V}_A)$, where $\bar{V}_A$ denotes the $m \times m$ submatrix of $\bar{V}$ whose columns are indexed by the elements of $A$, and $\tdet$ denotes the tropical determinant. More explicitly, if $A = \{a_1, \dotsc , a_m\}$ then 
\[
 (\tau_V)_A = \min_{\sigma \in S_m} \left( \bar{V}_{a_1 , \sigma(a_1)} + \bar{V}_{a_2 , \sigma(a_2)} + \dotsb + \bar{V}_{a_m , \sigma(a_m)} \right).
\]
It follows from this construction that the vector $\tau_V$ is a tropical Pl\"ucker vector, and that its underlying matroid is the principal transversal matroid $M'$ of the bipartite graph with vertex set $[n]$ and with an edge $(i,j)$ if $i \in I_j$ (see \cite{transversal}). The matroid polytopes $\Gamma$ and $\Gamma':= \Gamma(M')$ thus have the same vertex figure around the basis $B$. Moreover, if $\D$ denotes the regular matroid subdivision induced by $\tau_V$ on $\Gamma'$, the subdivision $\D_B$ induced by $\D$ on the vertex figure $\Gamma'_B = \Gamma_B$ is equal to the original subdivision $\mathcal{S}$ (see {\cite[Corollary 1.4.14]{kapranov} and \cite[Proposition 4]{joswig}}).

Given a polyhedral subdivision $\Sigma$ of a polytope $P$, let us denote by $\mathcal{I}(\Sigma)$ the graded poset of interior faces of $\Sigma$ ordered by reverse inclusion. The following proposition was proved for $m \leq 3$ and conjectured for general $m$ in a first version of \cite{joswig}. It was later proved in a second version of their paper \cite[Theorem 7]{joswig}. Our proof was obtained independently, and presents different ideas to the ones used in their approach.

\begin{proposition}\label{joswigconjecture}
 Assume $\mathcal{S}$ is a regular subdivision of the vertex figure $\Gamma_B \cong \convex\{ e_j + \Delta_{I_j} \mid j \in [n] \setminus B \}$ induced by the matrix $V \in \T^{B \times ([n]\setminus B)}$. Let $\D$ be the matroid subdivision induced by $\tau_V \in \T ^{\binom{[n]}{m}}$ on the matroid polytope $\Gamma'$. Then the posets $\mathcal{I}(\mathcal{S})$ and $\mathcal{I}(\D)$ are isomorphic. 
\end{proposition}
\begin{proof}
 Note that replacing the entries of $V$ which are equal to $\infty$ by large enough real numbers naturally produces an extension of the subdivisions $\mathcal{S}$ and $\mathcal{D}$ to the product of simplices $\Delta_{[n]\setminus B} \times \Delta_B$ and the hypersimplex $\Delta_{m,n}$, respectively. We can therefore assume that $\Gamma$ is the full hypersimplex $\Delta_{m,n}$, that $\Gamma_B \cong \Delta_{[n]\setminus B} \times \Delta_B$, and that $V$ is a real matrix.

 Now, we first prove that every facet of $\D$ contains the vertex $e_B$, and thus every facet of $\D$ intersects $\Gamma_B$ in a facet of the subdivision $\mathcal{S} = \D_B$. Note that a sufficiently small perturbation on the matrix $V$ produces a refinement on both subdivisions $\mathcal{S}$ and $\D$, so without loss of generality we can assume that $\mathcal{S}$ is a triangulation of $\Gamma_B$.
 
 If $\mathcal{S}$ is a triangulation of $\Gamma_B$ then, as discussed in Proposition \ref{localfvector} and Lemma \ref{fvectormixed}, the number of facets of $\mathcal{S} = \D_B$ is exactly 
\begin{equation*}
 \binom{n-2}{ n-m-1, \, 0,\, m-1} = \binom{n-2}{m-1}.
\end{equation*}
Each of these facets arises as the intersection of a facet of $\D$ with the vertex figure $\Gamma_B$. It was proved by Speyer in \cite[Theorem 6.1]{speyer} that any matroid subdivision of the hypersimplex $\Delta_{m,n}$ has at most $\binom{n-2}{m-1}$ facets, so this implies that the subdivision $\D$ has exactly $\binom{n-2}{m-1}$ facets, all of which contain the vertex $e_B$.

Now, every interior face of $\D$ can be obtained as an intersection of facets of $\D$, so all interior faces contain the vertex $e_B$. It follows that the map from $\mathcal{I}(\D)$ to $\mathcal{I}(\mathcal{S})$ sending the face $F$ to the face $F \cap \Gamma_B$ is an isomorphism, as desired.  
\end{proof}

Proposition \ref{joswigconjecture} states that the regular matroid subdivision $\D$ induced by $\tau_V$ is simply the ``cone'' from vertex $e_B$ over the subdivision $\mathcal{S}$, in such a way that all the facets of $\D$ contain the vertex $e_B$. Any matroid subdivision arising in this way will be called a \textbf{conical matroid subdivision}. A tropical linear space dual to a conical matroid subdivision will be called a \textbf{conical tropical linear space}. Conical tropical linear spaces are precisely those tropical linear spaces such that all their bounded faces lie in a single local tropical linear space.

\begin{example}
 Suppose $m=2$. In this case, the space of tropical Pl\"ucker vectors with finite coordinates agrees with the space of phylogenetic trees on $n$ leaves (see \cite{ss}). In fact, after modding out by its lineality space, any tropical linear space dual to a matroid subdivision of the hypersimplex $\Delta_{2,n}$ is homeomorphic to a tree with $n$ (unbounded) leaves. Each of its unbounded rays is dual to a facet of $\Delta_{2,n}$ of the form $x_i = 1$, so they are naturally labeled by the numbers $1,\dotsc,n$. The local tropical linear space around a basis $B = \{i,j\}$ is simply the unique path in the tree between leaves $i$ and $j$. It follows that the tropical linear space is conical if and only if there is a path between two of its leaves containing all internal vertices, that is, if and only if it is a \textbf{caterpillar tree}.
\end{example}

\begin{theorem}\label{conicalfvector}
 Any conical tropical linear space satisfies the $f$-vector conjecture.
\end{theorem}
\begin{proof}
 As discussed above, $L$ is a conical tropical linear space if and only if there exists some local tropical linear space $L_B$ containing all bounded faces of $L$. The result follows from Proposition \ref{localfvector}.
\end{proof}

Since any subdivision of the vertex figure $\Gamma_B$ can be refined to a triangulation, any conical matroid subdivision of the hypersimplex $\Delta_{m,n}$ can be refined to a conical matroid subdivision with exactly $\binom{n-i-1}{i-1} \binom{n-2i}{m-i}$ bounded faces of codimension $i$ (see Propositions \ref{localfvector} and \ref{fvectormixed}). This implies that any conical tropical linear space dual to a subdivision of $\Delta_{m,n}$ can be ``refined'' into a conical tropical linear space whose $f$-vector attains the upper bound predicted by the $f$-vector conjecture (Conjecture \ref{fvector}). However, there are tropical linear spaces attaining this upper bound which are not conical. If $m=2$, for example, any tropical linear space homeomorphic to a trivalent tree with $n$ leaves that is not a caterpillar tree also attains this upper bound, but it is not a conical tropical linear space. 

It was proved in \cite{joswig} that for any basis $B$ of the uniform matroid $U_{m,n}$, the resulting map $\tau: \R^{m \times (n-m)} \to \Dr_{m,n}$ defined by $V \mapsto \tau_V$ is a combinatorial embedding of the secondary fan of $\Delta_{[n] \setminus B} \times \Delta_{B}$ into the Dressian $\Dr_{m,n} \cap \, \R^{\binom{[n]}{m}}$, that is, it is an injective map between polyhedral fans preserving dimension and the inclusion relation between the cones. The union of the images of the different $\tau$ arising from all possible bases $B$ is precisely the set of tropical Pl\"ucker vectors corresponding to conical tropical linear spaces. It would be very interesting to understand more explicitly what part of the tropical Grassmannian $\TGr_{m,n}$ corresponds to conical tropical linear spaces in the case where $m \geq 3$.

\bibliographystyle{amsalpha}
\bibliography{../Bibliography/bibliography}

\end{document}